\documentclass[nobib]{tufte-handout}
\usepackage{amssymb,amsmath,amsthm}

\usepackage{hyperref}

\usepackage{graphicx} 
  \setkeys{Gin}{width=\linewidth,totalheight=\textheight,keepaspectratio}
  \graphicspath{{graphics/}} 
\usepackage{booktabs} 
\usepackage{units}    
\usepackage{multicol}

\newtheorem{thm}[subsection]{Theorem}
\newtheorem{defn}[subsection]{Definition}

\theoremstyle{definition}

\newtheorem{rmk}[subsection]{Remark}
\title{On Suslin Matrices and their connection to Spin groups}
 \date{}

\author{Vineeth Chintala}

\begin{document}

\begin{NoHyper}

\maketitle

\chapter {Introduction}

\vskip 5mm

William Clifford (1878) constructed his algebras in order to generalise Complex numbers and Quaternions to higher dimensions. Though Clifford algebras have found wide applications in geometry, quantum mechanics and applied mathematics, their abstract construction may not be easy to work with. Even in the simplest case for a hyperbolic quadratic form, where Clifford algebras are total matrix rings, there is no simple isomorphism describing their generators. In this paper, we will give a simple description of Clifford algebras 
using what are called Suslin matrices. The identities followed by these matrices will then be 
used to study the corresponding Spin groups. Conversely one will now be able to relate
some (seemingly) accidental properties of Suslin matrices to the geometry of Clifford algebras. 

Let $R$ be a commutative ring and $V$ be a free $R$-module.
This article is on the Clifford algebra of the hyperbolic module $H(V) = V \oplus V^*$ 
equipped with the quadratic form $q(x,f) = f(x)$.
These algebras are isomorphic to total matrix rings of size $2^n \times 2^n$, where
$\dim V = n$. \vskip 3mm

The representation of Clifford algebras in terms of Suslin matrices is faithful only when $\dim V$ is odd and we will mostly be concerned with this case.
With an explicit construction at hand, one can expect to derive some results on the Clifford algebra and its Spin group 
using simple matrix computations. 
For example, the identities followed by the Suslin matrices (in Equation \ref{eq1})
give us (as we will see) the involution on the Clifford algebra. 
Indeed given any element of the Clifford algebra we will be able to explicitly construct its 
conjugate matrix under the involution. As an application of Suslin matrices, we give a proof of
the following exceptional isomorphisms : $$Spin_4(R) \cong SL_2(R) \times SL_2(R), \hskip 1mm Spin_6(R) \cong SL_4(R).$$

\vskip 2mm
On the other hand, some key properties of Suslin matrices (see Theorem \ref{fundamental} and the
margin-notes in \ref{rmk00}) are re-derived here (with minimal computation) using the connection to Clifford algebras. 
These properties play the crucial role in the papers \cite{JR}, \cite{JR2} which study an action of the orthogonal group 
on unimodular rows. In addition, this connection to Clifford algebras explains \textit{why} the size of a Suslin matrix has to be of the form
$2^{n-1} \times 2^{n-1}$. One can go on and ask if there are other matrix-constructions, possibly of smaller sizes.
Indeed, it turns out that the size chosen by Suslin is the least possible.
This is an \emph{embedding} problem of  a quadratic space into 
the algebra of matrices - it belongs to a larger theme of finding natural ways of constructing
algebras out of quadratic spaces. It would be interesting to see similar constructions for other types of 
quadratic forms -  and for this reason alone, it might be worthwhile having this connection between Suslin matrices
and Clifford algebras. 

\vskip 3mm

We begin with a few preliminaries on Clifford algebras
and Suslin matrices (to make the article accessible) and the link between them.
The involution of the Clifford algebra is then described via some identities
followed by the Suslin matrices. It turns out that the behavior of the Suslin matrices (hence that of the
Spin groups) depends on whether the dimension of $V$ is odd or even. 
One realizes this after observing the pattern followed by the involution in the odd and even cases. This leads us to
Section 5 where the action of the Spin group on the Suslin space is studied in detail. 
The last part of the paper contains a few remarks on the Spin and Epin groups. 

\vskip 2mm
By a ring, we always mean a ring with unity.
In this paper, $V$ will always denote a free $R$-module where $R$ is any commutative ring
and $H(V) = V \oplus V^*$. From now on, we fix the standard basis for $V= R^n$ and identify $V$ with its dual $V^*$. 
One can then write the quadratic form on $H(V)$ as 
$$q(v,w) = v\cdot w^T = a_1b_1 + \cdots + a_nb_n.$$
for $v =(a_1,\cdots,a_n)$, $w =(b_1,\cdots,b_n)$.
\sidenote{References : For general literature on Clifford algebras and Spin groups over a commutative ring, the reader is referred to 
\cite{B1}, \cite{B2} of H. Bass. For a more detailed introduction to Suslin matrices (and their important connection 
to unimodular rows), a few references are Suslin's paper \cite{S}, 
the papers \cite{JR}, \cite{JR2} of Ravi Rao and Selby Jose, and T. Y. Lam's book \cite{L} (sections III.7, VIII.5). }

\vskip 3mm

{\bf Acknowledgements.} 
I would like to thank my advisor Prof. Ravi Rao for the many discussions we had while working on this paper.
His willingness to give his time so generously has been very much appreciated. 

\hskip 2mm Part of this work was done during a short enjoyable visit to Fields Institute (in Toronto),
which funded me to attend its Spring school on Torsors in 2013. 
\vskip 5mm

\section{\textbf{The Suslin Construction}}
\vskip 4mm

The Suslin construction gives a sequence of matrices whose size doubles at every step. Moreover each Suslin matrix $S$ has a conjugate Suslin matrix $\overline{S}$
such that $S\overline{S}$ and $S + \overline{S}$ are 
scalars (norm and the trace).
\vskip 2mm

Let us pause here to see the recursive process by which the Suslin matrix $S_n(v,w)$ of size 
$2^n \times 2^n$,
and determinant $(v \cdot w^T)^{2^{n-1}}$, is constructed 
from two vectors $v,w$ in $R^{n+1}$. 
\vskip 2mm

Let $v = (a_0, v_1)$, $w = (b_0, w_1)$ where $v_1$, $w_1$ are vectors in $R^n$. Define

 $$ S_ 0(a_0,b_0) = a_0, \hskip 2mm  
 S_1(v,w) =  \begin{pmatrix}
a_0     &v_1    \\
-w_1    & b_0 \end{pmatrix}$$
 and 
 $$ 
 S_n(v,w) =  \begin{pmatrix}
a_0I_{2^{n-1}}      &  S_{n-1}(v_1,w_1) \\\\
-S_{n-1} (w_1, v_1)^T        &  b_0I_{2^{n-1}}\end{pmatrix}.$$
\vskip 2mm
For $S_n = S_n(v,w)$, define $$\overline{S_n} := S_n(w,v)^T =  \begin{pmatrix}
b_0I_{2^{n-1}}      &  -S_{n-1}(v_1,w_1) \\\\    
S_{n-1} (w_1, v_1)^T        &  a_0I_{2^{n-1}}\end{pmatrix}.$$
\vskip 3mm

For example if $v=(a_1,a_2)$ and $w =(b_1,b_2)$, then 
\[ S(v,w) =  \begin{pmatrix}
a_1 & a_2\\
-b_2 & b_1  \\
\end{pmatrix} , \hskip 2mm
 \overline{S(v,w)} =  \begin{pmatrix}
b_1 & -a_2\\
b_2 & a_1  \\
 \end{pmatrix} .
$$
\vskip 2mm

For $v=(a_1,a_2,a_3)$ and $w =(b_1,b_2,b_3)$,
 $$ S(v,w) =  \begin{pmatrix}
a_1 & 0 & a_2 & a_3\\
0   & a_1 & -b_3 & b_2  \\
-b_2 & a_3 & b_1 & 0 \\
-b_3 & -a_2 & 0  & b_1 \end{pmatrix} , \hskip 2mm
 \overline{S(v,w)} =  \begin{pmatrix}
b_1 & 0 & -a_2 & -a_3\\
0   & b_1 & b_3 & -b_2  \\
b_2 & -a_3 & a_1 & 0 \\
b_3 & a_2 & 0  & a_1 \end{pmatrix} .
\]
\vskip 3mm
For $v = (a_1,a_2,a_3,a_4)$ and $w= (b_1,b_2,b_3,b_4)$, the corresponding $8 \times 8$ Suslin matrices are
\[S(v,w) =  \begin{pmatrix}
a_1 & 0 & 0 & 0     \hskip 2mm & a_2 & 0 & a_3 & a_4           \\
0  & a_1 & 0 & 0   \hskip 2mm  & 0   & a_2 & -b_4 & b_3                \\
0 & 0 & a_1 & 0     \hskip 2mm & -b_3 & a_4 & b_2 & 0          \\
0 & 0 & 0  & a_1   \hskip 2mm  & -b_4 & -a_3 & 0  & b_2         \\
\\
-b_2 & 0 & a_3 & a_4       &b_1 & 0 & 0 & 0         \\
0   & -b_2 & -b_4 & b_3      & 0  & b_1 & 0 & 0      \\
-b_3 & a_4 & -a_2 & 0        &0 & 0 & b_1 & 0          \\
-b_4 & -a_3 & 0  & -a_2        &0 & 0 & 0  & b_1  \\
\end{pmatrix} $$ and
$$ \overline{S(v,w)} =  \begin{pmatrix}
b_1 & 0 & 0 & 0   \hskip 2mm   & -a_2 & 0 & -a_3 & -a_4           \\
0  & b_1 & 0 & 0  \hskip 2mm   & 0   & -a_2 & b_4 & -b_3                \\
0 & 0 & b_1 & 0    \hskip 2mm  & b_3 & -a_4 & -b_2 & 0          \\
0 & 0 & 0  & b_1\hskip 2mm     & b_4 & a_3 & 0  & -b_2         \\
\\
b_2 & 0 & -a_3 & -a_4       &a_1 & 0 & 0 & 0         \\
0   & b_2 & b_4 & -b_3      & 0  & a_1 & 0 & 0      \\
b_3 & -a_4 & a_2 & 0        &0 & 0 & a_1 & 0          \\
b_4 & a_3 & 0  & a_2        &0 & 0 & 0  & a_1  \\
\end{pmatrix} .
\]
\vskip 2mm

Notice that $\overline{S_n(v,w)}$ is also a Suslin matrix. Indeed, $\overline{S_n(v,w)} = S_n(v',w')$ where 
$$\begin{pmatrix}
v'    \\
w'\end{pmatrix} = \begin{pmatrix}
(b_0, -v_1)    \\
(a_0, -w_1)\end{pmatrix}.$$
In particular, we have $\overline{S_0(a_0,b_0)} = b_0.$

\vskip 1mm
One can easily check that a Suslin matrix $S_n = S_n(v,w)$ satisfies the following properties: \\

\[ S_n\overline{S_n} = \overline{S_n}S_n = (v\cdot w^T) I_{2^n},\]
\[ \det S_n = (v \cdot w^T)^{2^{n-1}}.\]

Unless otherwise specified, 
we assume that $n>0$. Then the element $(v,w)$ is determined by its corresponding Suslin matrix $S_n(v,w)$. So we sometimes identify the element $(v,w)$ 
with $S_n(v,w)$ for $n>0$.
The set of Suslin matrices of size $2^n \times 2^n$ is an $R$-module 
under matrix-addition and scalar multiplication given by $rS_n(v,w) = S_n(rv,rw)$  for $r\in R$; Moreover when $n >0$,
the mapping $S_n \rightarrow \overline{S_n}$ is an isometry of the quadratic space $H(R^{n+1})$, preserving the quadratic form $q(v,w) = v\cdot w^T$.
 
\vskip2mm
 
In his paper \cite{S}, A. Suslin then describes a sequence of matrices $J_n \in M_{2^n}(R)$ by the recurrence formula

 \begin{equation*}
J_n = \begin{cases}
1      &\text{for $n = 0$}\\
\\
\begin{pmatrix}
J_{n-1}         & 0\\
0         & -J_{n-1}\end{pmatrix} &\text{for $n$ even}\\
\\
 \begin{pmatrix}
0         & J_{n-1}\\
-J_{n-1}        & 0 \end{pmatrix}  &\text{for $n$ odd}.
\end{cases}
\end{equation*}

\vskip 2mm

It is easy to check that $\det (J)$ $=1$ and 
\sidenote{We will simply write $J$ and $S$ (or $S(v,w)$) and drop the subscript when there is no confusion. When there is some confusion, remember that 
the subscript $r$ in $S_r$ is used to indicate 
that $S_r$ (or $J_r$) is a $2^r \times 2^r$ matrix.}
\[J^T = J^{-1} = (-1)^{\frac{n(n+1)}{2}}J.\]

\vskip 2mm
The matrix $J$ is 
$\begin{cases}
&\text{skew-symmetric for $n = 4k+1$ and $n = 4k+2$,}\\
&\text{symmetric for $n = 4k$ and $n = 4k + 3$.}
\end{cases}
$

\vskip 3mm 
One can show inductively that the forms $J$ satisfy the following identities  :
\sidenote{The even case in Equation (\ref{eq1}) can also be deduced from Lemma 5.3, ~\cite{S}, while the odd case follows from \textit{loc. cit.} only if $v \cdot w^t$ is not a zero-divisor.}

\begin{equation}\label{eq1}
 JS^TJ^T = 
\begin{cases}
S &\text{ for $n$ even,}\\
\overline{S} &\text{ for $n$ odd.}
 \end{cases}
\end{equation}

The above relations will be used  to describe the involution on the Clifford algebra. 
The Suslin matrices will now be used to give a set of generators of the Clifford algebra.

\section{\textbf{The link to Clifford algebras}}

Let $R$ be any commutative ring. Let $(V,q)$ be a quadratic space where $V$ is a free $R$-module of $\dim $ $V =n$, equipped with a quadratic form $q$.
The Clifford algebra $Cl(V,q)$ is the quotient of the tensor algebra $$T(V) = R \oplus V\oplus V^{\otimes 2} \oplus \cdots \oplus V^{\otimes n} \oplus \cdots$$ 
by the two sided ideal $I(V,q)$ generated by all $x\otimes x - q(x)$ with $x \in V$. 
\vskip 2mm

Thus $Cl(V,q)$ is an associative algebra (with unity) over $R$  with a linear map 
$i : V \rightarrow Cl(V,q)$ such that $i(x)^2 = q(x)$. The terms $x\otimes x$ and $q(x)$ appearing in the generators of $I(V,q)$ have degrees $0$ and $2$ in the grading of $T(V)$. 
By grading $T(V)$ modulo $2$ by even and odd degrees, it follows that the Clifford algebra has a $Z_2$-grading $Cl(V,q) =  Cl_0(V,q)\oplus Cl_1(V,q)$. 
\vskip 2mm

The Clifford algebra $Cl(V,q)$ has the following universal property : Given any associative algebra $A$ over $R$ and any linear map 
$j : V \rightarrow A$ such that 
$$ j(x)^2 = q(x) \text{ for all $x \in V$},$$
then there is a unique algebra homomorphism $f : Cl(V,q) \rightarrow A$ such that 
$f \circ i = j$.

\vskip 2mm 
 
Let $Cl$ denote the Clifford algebra of $H(R^n)$ equipped with the quadratic form
$$q(v,w) = v\cdot w^T.$$ It can be proved (see \cite{B1}, Ch. 5, Theorem 3.9) that 
 $$Cl \cong M_{2^n}(R).$$ 

\vskip 2mm

Let $\phi : H(R^n) \rightarrow M_{2^n}(R)$ be the linear map defined by 
$$\phi(v,w) =  \begin{pmatrix}
0                  & S_{n-1}(v,w) \\
\overline{S_{n-1}(v,w)}         & 0 \end{pmatrix} $$

where $\overline{S_{n-1}(v,w)} = S_{n-1}(w,v)^T$. 
\vskip 2mm

Since $\phi(v,w)^2 = q(v,w)I_{2^n}$, the map $\phi$ uniquely extends to 
an $R$-algebra homomorphism $$\phi : Cl \rightarrow M_{2^n}(R)$$ (by the universal property of Clifford algebras). In the next section, we will prove that $\phi$ is an isomorphism.

\subsection{Optimal embedding of Suslin matrices}
The Clifford algebra and the Suslin matrix are two different ways of constructing algebras out of the quadratic space $H(R^n)$.
One can ask if one can construct similar matrices but of different sizes - possibly smaller?
First we need a notion of an embedding of  a quadratic space in an algebra to
compare two constructions. 

\vskip 2mm

Let $(V,q)$ be  a quadratic space (i.e. $V$ is a free $R$-module and $q$ a non-degenerate quadratic form on $V$). 
Let $A$ be a faithful $R$-algebra 
($R\hookrightarrow A$). 
\vskip 2mm

\begin{defn} 
The quadratic space $(V,q)$ is said to be embedded in $A$ if $V \subseteq A$ and
there is an isometry $ \alpha : V \rightarrow V$ such that
$$v\alpha(v) = \alpha(v)v = q(v).$$
\end{defn}

\vskip 2mm

The object of our interest is the hyperbolic quadratic space $H(R^n)$ with the quadratic form $q(v,w) = v \cdot w^T$. By taking $\alpha : H(R^n) \rightarrow H(R^n)$ (for $n>1$)
to be the isometry $$\alpha : S_{n-1}(v,w) \rightarrow \overline{S_{n-1}(v,w)}$$ one sees that the Suslin construction $(v,w) \rightarrow S(v,w)$ is an embedding of 
$H(R^n)$ into $M_{2^{n-1}}(R)$.
\vskip 1mm
One can ask which properties of Suslin matrices are unique to its construction and which ones
are true in general for any other embedding into matrices. 

The size of a Suslin matrix corresponding to $(v,w) \in H(R^n)$ is $2^{n-1} \times 2^{n-1}$; the Clifford algebra (another embedding of $H(R^n)$) 
is isomorphic  to the algebra of 
$2^n \times 2^n$ matrices. Are there embeddings of smaller sizes? Theorem \ref{thm1} tells us that the size of the Suslin matrices is optimal.
\vskip 2mm

Let us analyze an embedding $\alpha :H(R^n) \hookrightarrow A$ into an associative algebra $A$. 
We will see that this places a strong restriction on the choice of $A$.
Consider the $R$-linear map $\phi : H(R^n) \rightarrow M_2(A)$ defined by 
$\phi(x) =  \bigl(\begin{smallmatrix}
0                  & x \\
\alpha(x)          & 0 \end{smallmatrix}\bigr) .$
\vskip 2mm

Since $\phi(x)^2 =  q(x)$, the map $\phi$ uniquely extends (by the universal property of Clifford algebras) to 
an $R$-algebra homomorphism $$\phi : Cl \rightarrow M_2(A).$$

\begin{thm}\label{thm1}
Let $\phi$ be defined as above; then $\phi$ is injective.
\end{thm}

\begin{proof}
Recall that we have an isomorphism of $R$-algebras $Cl \cong M_{2^n}(R)$ [\cite{B1}, Chapter 5, theorem 3.9], so it suffices to prove that the composite homomorphism 
$M_{2^n}(R) \rightarrow M_2(A)$ is injective.
Now every ideal of $Cl$ is of the form $ M_{2^n}(I)$. In particular $\ker(\phi) = M_{2^n}(I)$
for some ideal $I \subseteq R$. So to prove that $\phi$ is injective it is 
enough to observe that $I = 0$. But this is obvious since $\phi$ is $R$-linear and acts as identity on $R$.
\end{proof}

\vskip 2mm 

For the Suslin embedding $\alpha$ is the isometry $S(v,w) \rightarrow \overline{S(v,w)}$ defined for Suslin matrices in Section 2. 
Since the map $\phi$ is injective, it follows (by dimension arguments) that 
$\phi$ is an isomorphism for the Suslin embedding.
We will identify the elements $(v,w) \in H(R^n)$ with 
their images under the representation $\phi$. The following fundamental lemma of Jose-Rao in \cite{JR} 
is an easy consequence of the basic properties of Clifford algebras.
\vskip 5mm

\begin{thm}\label{fundamental}
 Let $X$ and $Y$ be Suslin matrices in $M_{2^{n-1}}(R)$. Then $XYX$ is also a Suslin matrix. Moreover $\overline{XYX} = \bar{X}\bar{Y}\bar{X}$.
\end{thm}

\begin{proof}
 Let $z_1, z_2 \in H(R^n)$. Then $$\langle z_1, z_2\rangle := z_1z_2 + z_2z_1 = (z_1 + z_2)^2 - z_1^2 - z_2^2$$ is an element in $R$. 
 Multiplying by $z_1$ we get $$z_1\langle z_1,z_2 \rangle= z_1^2z_2 + z_1z_2z_1.$$
 \vskip2mm
 
Since $z_1^2 = q(z_1)$, it follows that $z_1z_2z_1 \in H(R^n)$. Take 
$z_1 =  \bigl( \begin{smallmatrix}
0          & X\\
\overline{X}     & 0 \end{smallmatrix}\bigr) $ 
and $z_2 =  \bigl(\begin{smallmatrix}
0          & \overline{Y}\\
Y     & 0 \end{smallmatrix}\bigr) $.
Then $z_1z_2z_1 
=  \bigl(\begin{smallmatrix}
0                             &  XYX \\
\bar{X}\bar{Y} \bar{X} & 0 \end{smallmatrix}\bigr) .$
\end{proof}

\subsection{The basic automorphism of $Cl$}

The Clifford algebra $Cl = Cl_0 \oplus Cl_1$  has a `basic automorphism' given by
$$ x= x_0 + x_1 \rightarrow x' = x_0 - x_1. $$

Two automorphisms of $Cl$ are the same if they agree on the elements $(v,w) \in H(R^n)$.
Under the isomorphism $\phi$ in Section 3, the basic automorphism corresponds to conjugation by the matrix
$\lambda =  \bigl(\begin{smallmatrix}
1       &  0  \\
0        &  -1 \end{smallmatrix}\bigr).$ One can infer this by checking that
$\bigl(\begin{smallmatrix}
1       &  0  \\
0        &  -1 \end{smallmatrix}\bigr)\phi(v,w)\bigl(\begin{smallmatrix}
1       &  0  \\
0        &  -1 \end{smallmatrix}\bigr) = -\phi(v,w)$.
It is also easy to check that 
\sidenote{One needs the basic automorphism $\lambda$  to make a suitable adjustment for the grading of the Clifford algebra - for example when defining 
the norm or a `graded' conjugation. We will be concerned mostly with the Spin group
which lies in $Cl_0$ (on which conjugation by $\lambda$ is the identity map).}
\begin{equation*}
 \lambda^* : = J\lambda^TJ^T = 
\begin{cases}
\lambda &\text{ for $n$ even,}\\
-\lambda &\text{ for $n$ odd.}
 \end{cases}
\end{equation*}

\vskip 5mm

\section{\textbf{The involution on $Cl$}}

\vskip 4mm
We will identify $Cl$ 
with $M_{2^n}(R)$ via the isomorphism $\phi$. Recall that $\phi$ is defined on the elements $(v,w)$ as 
$$\phi(v,w) =  \begin{pmatrix}
0                  & S_{n-1}(v,w) \\
\overline{S_{n-1}(v,w)}         & 0 \end{pmatrix}.$$

The map $(v,w) \rightarrow (-v,-w)$ can also be viewed as an inclusion of $H(R^n)$ in the opposite algebra of $Cl$. By the universal property of the Clifford algebra, 
this map extends to an anti-automorphism of $Cl$. We will call this the standard involution on $Cl$.

\begin{thm}
 Let $M \in Cl \cong M_{2^n}(R)$. The standard involution $*$ is given by  

\[M^* = J_nM^TJ_n^T.\]
\end{thm}

\begin{proof}
That $*$ is an involution is clear since $J_n^T = J_n^{-1}$.
The elements $(v,w)$
generate the Clifford algebra.
Therefore, to prove that the above involution is the correct one, 
it is enough to check that its action on the matrices 
$$\phi(v,w) = \begin{pmatrix}
0                  & S_{n-1}(v,w) \\
\overline{S_{n-1}(v,w)}         & 0 \end{pmatrix} $$ is multiplication by $-1$. 
\vskip 4mm
Let $S'_n =  S_n(v',w') =\begin{pmatrix}
0                  & S_{n-1}(v,w) \\
 - \overline{S_{n-1}(v,w)}         & 0 \end{pmatrix} $ where $v' = (0,v)$ and $w'= (0,w)$. 
Then $$\phi(v,w) = \lambda S'_n = -S'_n\lambda.$$ 
\vskip 2mm 
Moreover observe that
 $\overline{S'_n} =  -S'_n.$ It follows by the identities in Equation (\ref{eq1}) that 

\begin{equation*}
{S'_n}^* =
\begin{cases}
S'_n &\text{ for $n$ even,}\\
-S'_n &\text{ for $n$ odd.}
\end{cases}
\end{equation*}

Therefore for $n$ even, we have
$$\phi(v,w)^* = {S'_n}^* \lambda^* = S'_n\lambda = -\lambda S'_n.$$

And for $n$ odd, we have
$$\phi(v,w)^* = {S'_n}^* \lambda^* = (-S'_n)(-\lambda) = -\lambda S'_n.$$

Hence for any $n$, 
\[\phi(v,w)^* = -\phi(v,w). \qedhere\]
\end{proof}

Given an element of the Clifford algebra, 
one can now compute its conjugate under the involution $*$. 
One can spot two different patterns of the involution depending on the parity of $n$. 
To see this, let us write $M =   \bigl(\begin{smallmatrix}
A          & B\\
C          & D \end{smallmatrix}\bigr) $ as a $2 \times 2$ matrix and 
analyze its conjugate in terms of its blocks.
Then one can compute $M^*$ inductively as follows:
(For $n=0$, the involution is the identity map on $R$.)
 
 \begin{equation}\label{eq2}
  \begin{pmatrix}
A          & B\\
C          & D \end{pmatrix} ^* = 
\begin{cases}
  \begin{pmatrix}
D^*          & -B^*\\
-C^*          & A^* \end{pmatrix}  &\text{for $n$ odd,}\\
\\
 \begin{pmatrix}
A^*          & -C^*\\
-B^*          & D^* \end{pmatrix}  &\text{for $n$ even.}
\end{cases}
\end{equation}
\vskip 2mm 

With this the setup is complete for us to 
analyze and compute the Spin groups.

\vskip 5mm
\subsection{ Spin group}
\vskip 3mm

The Clifford algebra is a $Z_2$-graded algebra $Cl = Cl_0 \oplus Cl_1$. Under the isomorphism $\phi$, the elements of $Cl_0$ correspond to matrices of the form 
$\bigl(\begin{smallmatrix}
g_1      &  0  \\
0        &  g_2 \end{smallmatrix}\bigr)$. 
\vskip 2mm

The following groups are relevant to our discussion : 
    $$U_{2n}^0(R) := \{x \in Cl_0 \, | \,  xx^* = 1\}.$$
    $$ Spin_{2n}(R):= \{x \in U_{2n}^0(R) \,| \, xH(R^n)x^{-1} = H(R^n)\}.$$  
\vskip 2mm
Let $(g_1,g_2) \in Spin_{2n}(R).$ We have by the identities in Equation (\ref{eq2}),

$$
  \begin{pmatrix}
g_1      &  0 \\
0        &  g_2 \end{pmatrix} ^* = 
 \begin{cases}
   \begin{pmatrix}
g_1^*      &  0  \\
0        &  g_2^* \end{pmatrix}  &\text{for $n$ even,}\\
\\
  \begin{pmatrix}
g_2^*      &  0\\
0        &  g_1^* \end{pmatrix}   &\text{for $n$ odd. }
\end{cases}
$$
\vskip 2mm
Where there is no confusion possible, 
we will write $(g_1, g_2)$ for the diagonal matrix 
$ \bigl(\begin{smallmatrix}
g_1      &  0  \\
0        &  g_2 \end{smallmatrix}\bigr) .$ 

\vskip 5mm

\section{ \textbf{The action of the Spin group on the space of Suslin matrices : When $n$ is odd}}
\vskip 4mm

In the introduction of \cite{JR}, the authors shared the insight of their referee (of \cite{JR2}) that one should be able to construct 
a subgroup $G_r(R)$ of $GL_{2^r}(R)$ which is defined by the property that $g \in G_r(R)$ if $gS(v,w)g^*$ is a Suslin matrix for all Suslin matrices $S \in M_{2^r}(R)$. 
From this action of $G_r(R)$ on the space of Suslin matrices $S(v,w)$, 
consider the subgroup $SG_r(R)$ consisting of $g \in G_r(R)$ preserving the norm $v \cdot w^T$ for all pairs $(v,w)$. 
The referee guessed that the group $SG_r(R)$ is isomorphic to the Spin group. The first step is to realize that the behaviour depends on the parity of $\dim V$.
\vskip 1mm
In this section, we construct such a group $G_r(R)$ and prove that there is such an isomorphism between 
$SG_r(R)$ and the corresponding Spin group when $n = \dim V$ is odd (i.e. $r = n-1$ is even). This is achieved by taking $g \rightarrow g^*$ to be the restriction of the 
involution on $M_{2^r}(R)$ defined earlier. In the rest of the section, $n = \dim V$ is odd. 
\vskip 2mm

Let $M =  \bigl(\begin{smallmatrix}
A          & B\\
C          & D \end{smallmatrix}\bigr)\in M_{2^n}(R)$.
When $n$ is odd, we have 
 $M^* = 
  \bigl(\begin{smallmatrix}
D^*          & -B^*\\
-C^*          & A^* \end{smallmatrix}\bigr) .$ Therefore for $(g_1,g_2) \in U_{2n}^0(R)$, we have
$$(g_1,g_2)^* = (g_2^*,g_1^*).$$  In addition, since $(g_1,g_2) \in U_{2n}^0(R)$ has unit norm, 
i.e. $(g_1, g_2)(g_1,g_2)^* = 1,$ it follows that
$$g_2 = g_1^{*^{-1}}.$$

\vskip 2mm
Now $Spin_{2n}(R)$ is precisely the subgroup of $U_{2n}^0(R)$ which stabilizes $H(R^n)$ under 
conjugation.
This means that if $(g, g^{*^{-1}}) \in Spin_{2n}(R)$ and $S \in M_{2^{n-1}}(R)$ is any Suslin matrix, 
then there exists a Suslin matrix $T \in M_{2^{n-1}}(R)$ such that 
$$(g, g^{*^{-1}})
 \begin{pmatrix}
0                  & S \\
\overline{S}           & 0 \end{pmatrix}
(g^{-1}, g^*) = 
 \begin{pmatrix}
0                  & T \\
\overline{T}          & 0 \end{pmatrix} $$
i.e.,
$$
 \begin{pmatrix}
0                  & gSg^* \\
g^{*^{-1}}\overline{S}g^{-1}           & 0 \end{pmatrix}  = 
 \begin{pmatrix}
0                  & T \\
\overline{T}         & 0 \end{pmatrix} .$$
\vskip 3mm

Hence if $S \in M_{2^{n-1}}(R)$ is a Suslin matrix and $(g, g^{*^{-1}}) \in Spin_{2n}(R)$, then $gSg^*$ is also a Suslin matrix.
\vskip 2mm

\begin{rmk}
The space of Suslin matrices is nothing
but the quadratic space $H(R^n)$. 
For simplicity, we will write $S$ instead of $S_{n-1}$ and 
when we say ``for all Suslin matrices $S$'', we really mean ``for all Suslin matrices $S \in M_{2^{n-1}}(R)$''.
\end{rmk}

\subsection{The action on the quadratic space $H(R^n)$}

Let $g\bullet S = gSg^*$.
Consider the group $$G_{n-1}(R) = \{\ g\in GL_{2^{n-1}}(R) \ | \text{ $g\bullet S$ 
is a Suslin matrix $\forall$ Suslin matrices $S$}\}.$$ 
One has the homomorphism
$$\chi : Spin_{2n}(R) \rightarrow G_{n-1}(R)$$ given by $(g,{g^*}^{-1}) \rightarrow g.$
\vskip 1mm

In general, this homomorphism is not surjective, but one can expect that it is the case on the subgroup of $G_{n-1}(R)$ which preserves the quadratic form $v \cdot w^T$.

For this we introduce a length function  on the space of Suslin matrices : for $S = S_{n-1}(v,w)$, define $$l(S) : = S\overline{S} = v\cdot w^T.$$ 
Let $$SG_{n-1}(R):= \{\  g \in G_{n-1}(R)\ |  \text{ $l(g \bullet S) = l(S)$ $\forall$ Suslin matrices $S$}\}.$$ 
Suppose $g \in SG_{n-1}(R)$. Then $l(gg^*) = 1.$ 
One expects $(g, g^{*^{-1}})$ to be an element of the Spin group. 
\vskip 2mm
\begin{thm}
 The homomorphism $\chi : Spin_{2n}(R) \cong SG_{n-1}(R)$ is an isomorphism.
\end{thm}
\begin{proof}
 We first prove that if $g \in SG_{n-1}(R)$ then ${g^*}^{-1} \in G_{n-1}(R).$
Let $T = g\bullet S$. If $l(S) = l(T) = 1$, then $\overline{T} = T^{-1} = g^{*^{-1}}\bullet \overline{S}.$
Write a general Suslin matrix as a linear combination of unit-length Suslin matrices. 
By linearity of the action $\bullet$ 
it follows that $\overline{T} = g^{*^{-1}}\bullet \overline{S}$ for a general Suslin matrix $S$. It follows that we can define a homomorphism $SG_{n-1}(R) \rightarrow Spin_{2n}(R)$ 
by $g \rightarrow (g,g^{*^{-1}})$ and one checks easily that this is an inverse of $\chi$.
\end{proof}

\vskip 3mm

The assumption $l(g\bullet S)= l(S)$ is simply a translation of the definition of the Spin group in 
terms of Suslin matrices. It does not give us much insight into the action $\bullet$. 
A simpler equivalent criterion is the following: 
for $g \in G_{n-1}(R)$, we have 
$l(g\bullet S) = l(S)$ for all Suslin matrices $S$ if and only if $l(gg^*) = 1$. 
\vskip 1mm
This will be proved by replacing the length function which is defined on Suslin matrices with a `norm' which makes sense 
for any element $g \in G_{r-1}(R)$. The Spin group then corresponds to the subgroup consisting of elements which have unit norm.
In particular, the following theorem implies that $$ SG_{n-1}(R) = \{\  g \in G_{n-1}(R)\ | \ l(gg^*) =1\}.$$

\vskip 4mm

\begin{thm} \label{main} 
Let $g \in G_{n-1}(R)$. Then $l(g\bullet S) = l(gg^*)l(S)$ for all Suslin matrices $S$.
\end{thm}
\begin{proof}
\textbf{Case 1 : $l(gg^*) = 1.$}
\vskip 2mm
It is enough to show that $\overline{g \bullet S} = g^{*^{-1}} \bullet \overline{S}$ which in turn 
implies that $l(g\bullet S) = l(S)$. We will first prove this for a subset of Suslin matrices and then show that our 
chosen set spans the total space.
 
\vskip 1mm
 \textbf{Step 1 :}
If $X,Y$ are Suslin matrices then it follows by Theorem \ref{fundamental} that $$l(XYX) = l(X)^2l(Y).$$
Take $X = g\bullet S$ and $Y= (gg^*)^{-1}$. Then $XYX = g \bullet S^2$.
In addition we have $$gg^* \bullet (gg^*)^{-1} = gg^*.$$

Since $gg^*$ is a Suslin matrix and $l(gg^*) = 1$ 
(by the hypothesis) it follows that its inverse $Y$ is also a Suslin matrix with $l(Y) = 1$. Therefore 
$l(XYX) = l(X)^2,$ i.e.
$$l(g\bullet S^2) = l(g\bullet S)^2.$$
By induction we get that $$l(g\bullet S^{2^k}) = l(g\bullet S)^{2^k} \text{ for all $k \geq 1.$}$$
Suppose $l(S) = 1$.
From the definition of the length function, we have $$l(g\bullet S)^{2^{n-2}} = \det (g\bullet S) = \det (gg^*) \det (S) = 1.$$ 
Therefore 
\begin{equation}\label{eq3}
 l(g \bullet S^{2^k}) = 1, \text{ for all $k\geq n -2$.}
\end{equation}
\textbf{Step 2 :} Let \[W = \{\sum\limits_{i=1}^k  r_iT_i   \ | \text{$r_i \in R$ and $T_i = S_i^{2^{m_i}}$ 
for some $S_i$ with $l(S_i) =1$ and $m_i > n$}.\}\]
We will first prove the theorem for $W$ and then show that it spans the whole space.

\vskip 3mm
\noindent If $T = S^{2^m}$ with $m > n$ and $l(S) = 1$, then we have by Equation (\ref{eq3}) that 
$$\overline{g \bullet T} = ({g \bullet T} )^{-1} = g^{*^{-1}} \bullet \overline{T}.$$ 
For any $T \in W$ it follows by the linearity of the action $\bullet$ that 
$$g^{*^{-1}} \bullet \overline{T} =  g^{*^{-1}}\bullet (r_i \overline{T_i} + \cdots + r_k\overline{T_k}) = \overline{g \bullet T}.$$
Therefore $l(g \bullet T) = l(T)$ for all $T \in W.$ 
\vskip 1mm

\textbf{Step 3 :}
Let $\{e_i\}$ be the standard basis of $R^n$ and $\{f_i\}$ be its dual basis. We will now show that
$W$ contains the generators $S_{n-1}(e_i, 0)$ and $S_{n-1}(0,f_i)$, hence the whole space.
Let $$E_i  = S_{n-1}(e_i, 0), \hskip 2mm  F_i  = S_{n-1}(0,f_i).$$ 
\vskip 2mm

Let $X =  \bigl(\begin{smallmatrix}
1         & A\\
-\overline{A}          & 0 \end{smallmatrix}\bigr) $ for some Suslin matrix $A$ with $A\overline{A} =I$.
Then $X$ is also a Suslin matrix and $X^2 = \bigl(\begin{smallmatrix}
0       & A\\
-\overline{A}          & -1 \end{smallmatrix}\bigr) $ and 
$ X^3 = - I$. Therefore $X^{2^k} = X^2$ or $X^{2^k} = -X$, depending on whether $k$ is odd or even. 
It follows by the above discussion that $X \in W$. For example, one can take $ X = E_i + F_i + E_1$ with $i \neq1$. 
Also, if $i \neq j$ and $i,j \neq 1$, then $E_i + E_j + F_j + E_1$ is an example.
\vskip 2mm

Let $i \neq j $ and $i,j \neq 1$. 
Then $E_i \in W$ since $$ E_i = (E_i + E_j + F_j + E_1) - (E_j + F_j + E_1)$$ is a linear combination of matrices in $W$.
Similarly $F_i \in W$ since $$ F_i = (F_i + E_j + F_j +E_1) - (E_j +F_j +E_1).$$ 
And finally $ E_1 =  (E_i + F_i + E_1) - (E_i + F_i)$ and $F_1 = I - E_1$ also lie in $W$.  
 
 \vskip 3mm
 
\textbf{Case 2 : $l(gg^*) = a.$}
\vskip 2mm

Clearly $a$ has to be invertible since $g \in G_{n-1}(R)$.
Suppose there is an $x \in R$ such that $x^2 = a^{-1}.$ 

\vskip 2mm

Take $h = xg$. Then $l(hh^*) =1$ and
by Case 1, we have $$l(h\bullet S) = l(hh^*)l(S)$$ for any Suslin matrix $S$.
\vskip 2mm

We also have $l(h\bullet S) = x^2\cdot l(g\bullet S)$, hence $l(g\bullet S) = l(gg^*)l(S).$
\vskip 2mm

Now suppose $x^2 = a^{-1}$ has no solutions in $R$. Then one has the identity $$l(g\bullet S) = l(gg^*)l(S)$$ over the ring $\frac{R[x]}{(x^2 - a^{-1})}$; since each term of the equation 
lies in $R$, the theorem is proved in this case too.
\end{proof}

Let $R^\times$ denote the group of units in $R$.
\vskip 5mm
\begin{thm}\label{rmk}
Define $d : G_{n-1}(R) \rightarrow R^\times$ as 
\[d(g) : = l(gg^*).\]
Then $\ker (d) = SG_{n-1}(R) \cong Spin_{2n}(R).$
\end{thm}
\begin{proof}
Let $g,h \in G_{n-1}(R)$. As a consequence of Theorem \ref{main}, we have  
$$d(gh) = l(ghh^*g^*) = l(gg^*)l(hh^*) = d(g)d(h).$$ Thus $d$ is a group homomorphism and 
$\ker (d) = SG_{n-1}(R) \cong Spin_{2n}(R).$
 \end{proof}

\section{ \textbf{When $n$ is even:}}

\vskip 5mm

Let $(g_1, g_2) \in Spin_{2n}(R)$. By definition if $S \in M_{2^{n-1}}(R)$ is a Suslin matrix, 
then there exists a Suslin matrix $T$ such that 
$$(g_1,g_2)
 \begin{pmatrix}
0                  & S \\
\overline{S}           & 0 \end{pmatrix}
(g_1^{-1}, g_2^{-1}) = 
 \begin{pmatrix}
0                  & T \\
\overline{T}           & 0 \end{pmatrix} $$
i.e.,
$$
 \begin{pmatrix}
0                  & g_1Sg_2^{-1} \\
g_2\overline{S}g_1^{-1}           & 0 \end{pmatrix}  = 
 \begin{pmatrix}
0                  & T \\
\overline{T}          & 0 \end{pmatrix} .$$

\noindent Let us now consider the group homomorphism $\chi : Spin_{2n}(R) \rightarrow GL_{2^{n-1}}(R)$ given by 
$$\chi(g_1,g_2) = g_1.$$

\noindent In this case the projection $\chi$ has a nontrivial kernel. Let $(1,g)\in$ $\ker(\chi)$. Taking $S=1$, one sees that $g$ is also a Suslin matrix.
Now both $g$ and $Sg^{-1}$ are Suslin matrices for any Suslin matrix $S$. 
In the case when $n > 2$, one can then conclude by a simple computation that $g = uI$ where $u \in R$ and $u^2=1$ 
(see \cite{JR}, Lemma 3.1).

\vskip 2mm

\noindent Define $\mu_2 := \{u \in R \,| \, u^2 = 1\}.$
Then one has an following exact sequence when $n > 2$ and $n$ even :
\begin{equation}
 1 \rightarrow \mu_2 \rightarrow Spin_{2n}(R)\rightarrow  GL_{2^{n-1}}(R). 
\end{equation}
\vskip 3mm

\textit{Question} : Let $SG'_{n-1}(R) = \chi (Spin_{2n}(R))$. For what conditions on $R$ is there a homomorphsim
$f : SG'_{n-1}(R) \rightarrow Spin_{2n}(R)$ such that $f \circ \chi = Id$?
\vskip 2mm

Recall that when $n$ is even,  $(g_1,g_2)^* = (g_1^*,g_2^*).$
From this one can immediately compute $Spin_4(R)$.
Firstly any $2 \times 2$ matrix is of the form $ S(v,w) =  \bigl(\begin{smallmatrix}
a_1 & a_2\\
-b_2 & b_1  \\
\end{smallmatrix}\bigr) $ for $v = (a_1, a_2)$, $w=(b_1, b_2)$.
Therefore $g_1Sg_2^{-1}$ is a Suslin matrix for any pair $(g_1,g_2)$ and $S\in M_2(R)$.
Moreover it is clear from the definition of the involution that the norm is precisely
the determinant for $2 \times 2$ matrices.
 $$(g_1, g_2)(g_1^*, g_2^*) = (\det g_1, \det g_2).$$ 
Therefore $$Spin_4(R) = SL_2(R) \times SL_2(R).$$
For another proof, see [\cite{K}, Chapter V, $\S$ 4.5].
\vskip 5mm

\begin{rmk}
 An immediate corollary is that $Spin_4(R)$ acts transitively on the set of $2 \times 2$ Suslin matrices $S(v,w)$ of unit length (i.e. $v\cdot w^T = 1$).  
The set of unit length Suslin matrices is nothing but $SL_2(R)$ and the action is given by $S \rightarrow g_1Sg_2^{-1}$ for $(g_1,g_2) \in Spin_4(R)$ and $S\in SL_2(R)$.

\end{rmk}

\vskip 5mm
\section{\textbf{Computing $Spin_6(R)$}}

\vskip 3mm

In this section, we compute $Spin_6(R)$ using Suslin matrices. For another proof of the following result, see [\cite{K}, Chapter V, $\S$ 5.6].
\vskip 3mm
\begin{thm} 
$Spin_6(R) \cong SL_4(R).$ 
\end{thm}
\vskip 2mm
\begin {proof}
 
We will prove that $SG_2(R) \cong SL_4(R)$.

\vskip 2mm

First it will be shown that given any $4 \times 4$ matrix $M$, the product $MSM^*$ is 
a Suslin matrix whenever $S$ is a Suslin matrix.
\vskip 3mm

Write $S = \bigl(\begin{smallmatrix}
a         &  S_1\\
-\overline{S}_1         & b \end{smallmatrix}\bigr) $ 
and
$ M = \bigl(\begin{smallmatrix}
A          & B\\
C          & D \end{smallmatrix}\bigr) .$ 
Since $M$ is a $4 \times 4$ matrix, we have
$$M ^* =  \big(\begin{smallmatrix}
A^*          & -C^*\\
-B^*          & D^* \end{smallmatrix}\bigr) .$$
\vskip 2mm

To conclude that $MSM^*$ is a Suslin matrix, it is enough to check that both
$M  \bigl(\begin{smallmatrix}
a         &  0\\
0         & b \end{smallmatrix}\bigr)  M^*$ and 
$M  \bigl(\begin{smallmatrix}
0       &  S_1\\
-\overline{S}_1         & 0 \end{smallmatrix}\bigr) M^*$ are Suslin matrices.
\vskip 3mm

We have 

$$M  \begin{pmatrix}
a         &  0\\
0         & b \end{pmatrix}  M^* = 
  \begin{pmatrix}
aAA^* - bBB^*        &  -aAC^* + bBD^*\\
aCA^*- bDB^*        &  -aCC^* + bDD^*\end{pmatrix} $$ 

and

$$M  \begin{pmatrix}
0       &  S_1\\
-\overline{S}_1         & 0 \end{pmatrix} M^* =
  \begin{pmatrix}
-B\overline{S}_1A^* - AS_1B^*      &  B\overline{S}_1C^* + AS_1D^*\\
-D\overline{S}_1A^* - CS_1B^*         & D\overline{S}_1C^* + CS_1D^* \end{pmatrix} .$$

Recall that any $2 \times 2$ matrix $X =  \bigl(\begin{smallmatrix}
x       &  y\\
z       & w \end{smallmatrix}\bigr)  $ is a Suslin matrix and $X^* = \overline{X} =  \bigl(\begin{smallmatrix}
w      &  -y\\
-z       & x \end{smallmatrix}\bigr)  $. 
Also, $X^* + X$ are $XX^*$ are scalar matrices. 
\vskip 2mm

A $4 \times 4$ matrix $ \bigl(\begin{smallmatrix}
X        & Y\\
Z          & W \end{smallmatrix}\bigr) $ is a Suslin matrix if and only if
$X, W$ are scalar matrices and $Y = -Z^*$.
Therefore $MSM^*$ is a Suslin matrix.
\vskip 2mm

It remains to show that if $g \in SG_2(R)$, then $\det g = 1$.
By definition $l(S)^2 = \det(S)$ for any Suslin matrix $S \in M_4(R)$. One should be able to prove from this that $\det g = l(gg^*) = 1.$  
\vskip 3mm

One can also check by hand that for $4 \times 4$ matrices, $l(MM^*) = \det M$.
We have $$M M^* = 
  \begin{pmatrix}
AA^* - BB^*        &  -AC^* + BD^*\\
CA^*- DB^*        &  -CC^* + DD^*\end{pmatrix} $$  and 
\[l(MM^*) = AA^*DD^* + BB^*CC^* + AC^*DB^* + BD^*CA^*. \qedhere \]
\end{proof}

\vskip 3mm 

\begin{rmk}
 Suppose $2$ is not a zero divisor in $R$. 
 Then a matrix $M \in M_4(R)$ is a Suslin matrix if and only if $M = M^*$, 
 i.e. the set of Suslin matrices is 
 the Jordan algebra consisting of self-adjoint elements in $M_4(R)$.
 In particular
 $MSM^*$ is a Suslin matrix for any $M \in M_4(R)$. 
 Observe that the set 
 $\{gg^* |\ g \in SL_4(R)\}$ is precisely the orbit of the identity element under the action
 of the Spin group. Therefore $Spin_6(R)$ acts transitively on the set of elements of length 
 $1$ if and only if every self-adjoint matrix in $SL_4(R)$ can be factorized as $S = gg^*$ 
 for some $g \in SL_4(R).$ 
 \vskip 2mm
 We have already seen that the Spin group 
 acts transitively in the case $n=2$.
 A related open question is whether the Orthogonal group acts transitively on the unit sphere (set of $(v,w)$ such that
 $v \cdot w^T =1$) for any
 commutative ring $R$. In \cite{S} (see Lemma 5.4) it has been proved that this is indeed the case when 
 $n =4$. The question is open for other dimensions. 
\end{rmk}

\vskip 5mm

\section{\textbf{ The group $Epin_{2n}(R)$}}
\vskip 4mm

Let $\partial$ denote the permutation $(1 \ n+1)...(n \ 2n)$ corresponding to the form 
$ \bigl(\begin{smallmatrix}
0          & I_n\\
I_n     & 0 \end{smallmatrix}\bigr) .$
We define for $1 \leq i \neq j \leq 2n$, $z\in R$,
$$oe_{ij}(a) = I_{2n} + ae_{ij} - ae_{\partial(j)\partial(i)}. $$

It is clear that when $a \in R$ all these matrices belong to $O_{2n}(R)$. We call them the elementary orthogonal matrices over $R$
and the group generated is called the elementary orthogonal group $EO_{2n}(R).$

\vskip 2mm
By definition there is a map $\pi : Spin _{2n}(R) \rightarrow O_{2n}(R)$ given by
$$\pi(g) : v \rightarrow gvg^{-1} \text{ for $g \in Spin_{2n}(R)$}.$$
\vskip 2mm

We denote by $Epin_{2n}(R)$ the inverse image of $EO_{2n}(R)$ under the map $\pi$. 

\vskip 2mm

Let $V = R^n$ with standard basis $e_1,\cdots, e_n$ and dual basis $f_1,\cdots, f_n$ for $V^*$. 
In terms of Suslin matrices, we have
$$e_i =  \begin{pmatrix}
0          & S_{n-1}(e_i, 0)\\
\overline{S_{n-1}(e_i, 0)}    & 0 \end{pmatrix} $$  
and
$$f_i =  \begin{pmatrix}
0          & S_{n-1}(0, f_i)\\
\overline{S_{n-1}(0, f_i)}    & 0 \end{pmatrix} $$

\vskip 5mm
\begin{rmk}
Observe that if $ \bigl(\begin{smallmatrix}
g_1          & 0\\
0     & g_2 \end{smallmatrix}\bigr) $ is an element of the Spin group, then so is $ \bigl(\begin{smallmatrix}
g_2         & 0\\
0     & g_1 \end{smallmatrix}\bigr) $.
We have
$$(g_2,g_1) = u\cdot (g_1,g_2)\cdot u$$
where $u =  \bigl(\begin{smallmatrix}
0          & I_n\\
I_n    & 0 \end{smallmatrix}\bigr)$
Below, we get a different basis of $Epin_{2n}(R)$ by conjugating its elements with $u$.
\end{rmk}

\vskip 5mm
\begin{rmk}\label{rmk2}
It can be proved 
(see \cite{B2}, section 4.3) that $Epin_{2n}(R)$ is generated by elements of the form $1 + ae_ie_j$ and  $1 + af_if_j$ with $a \in R$,
$1 \leq i,\ j \leq n$, $i \neq j$. Then it follows from the above symmetry of the Spin group that elements
$u(1 + ae_ie_j)u$ and  $u(1 + af_if_j)u$ also generate $Epin_{2n}(R)$ for $a \in R$. 
Therefore $Epin_{2n}(R)$ is generated by elements of the type $$1 + ae'_ie'_j, \hskip 2mm  1 + af'_if'_j$$ where

$$e'_i = e_iu =  \begin{pmatrix}
S_{n-1}(e_i, 0)        & 0\\
 0   & \overline{S_{n-1}(e_i, 0)}\end{pmatrix} $$
 and
 
 $$f'_i = f_iu =  \begin{pmatrix}
S_{n-1}(0, f_i)        & 0\\
 0   & \overline{S_{n-1}(0, f_i)}\end{pmatrix}.$$
\end{rmk}
 \vskip 2mm
 
In the rest of the section we will work with the case where n is odd. Recalling Remark \ref{rmk} we have 
$$Spin_{2n}(R) \cong SG_{n-1}(R):= \{\  g \in G_{n-1}(R) | \ l(gg^*) = 1\},$$ where the isomorphism is the projection 
$\chi :(g, {g^*}^{-1} ) \rightarrow g$.
\vskip 1mm

Let $EG_{n-1}(R)$ denote the image of $Epin_{2n}(R)$ under the above isomorphism. 
Remark \ref{rmk2} tells us that the group $EG_{n-1}(R)$ is generated by elements
of the type $$1 + aE_iE_j, \hskip 2mm 1+ aF_iF_j$$ where 
$$E_i = S_{n-1}(e_i, 0), \hskip 2mm  F_i = S_{n-1}(0,f_i).$$
\vskip 4mm

Notice that $E_1 = \bigl(\begin{smallmatrix}
I          & 0\\
0     &  0\end{smallmatrix}\bigr) $ and $F_1 = \bigl(\begin{smallmatrix}
0         & 0\\
0     & I \end{smallmatrix}\bigr)$. For $i \neq 1$, the element $E_i$ is of the form $\bigl(\begin{smallmatrix}
0         & X_i\\
-\overline{X_i}     & 0 \end{smallmatrix}\bigr)$ for some Suslin matrix $X_i$ and $l(E_i) = X_i\overline{X_i} =0$.

\vskip 2mm
The elements $E_i$, $F_i$ satisfy the properties in Table $\ref{commutator}$ :

\begin{margintable}[1pt]
\caption{$i \neq1$ }
  \begin{tabular}{ll}
    \toprule \\
    i)  $ E_1^2 = E_1$ & $ F_1^2 = F_1$\\
    \midrule \\
    ii)  $\overline{E_i} = - E_i$ & $\overline{F_i} = - F_i$ \\
     \hskip 3mm $E_i^2 = 0$ & $F_i^2 = 0$ \\
   \midrule \\
   iii) $E_iE_1 = F_1E_i$ & $F_iE_1= F_1 F_i$ \\

 \hskip 4mm$E_1E_i = E_iF_1$ & $E_1F_i = F_iF_1$ \\

 \hskip 4mm $E_iE_1 + E_1E_i = E_i$ & $F_iE_1 + E_1F_i = F_i$\\

         \bottomrule
  \end{tabular}
  \label{commutator}
\end{margintable}

\subsection{Generators of $EG_{n-1}(R) \cong Epin_{2n}(R)$}
It follows from properties ii) and iii)  of Table \ref{commutator} that
$E_1E_iE_1 = 0 = E_iE_1E_i$. Similarly we have $F_1F_iF_1 = 0 = F_iF_1F_i$.

Using this one can prove the following commutator relations :
\[1 + aE_iE_j = [1+ aE_iE_1, 1 + E_1E_j] \ \ i \neq 1, j \neq 1,\] 
\[1 + aF_iF_j = [1+ aF_iF_1, 1 + F_1F_j]  \ \ i,j \neq 1, j \neq 1.\]

\vskip 2mm

It follows from  Remark \ref{rmk2} and the above commutator relations that
$EG_{n-1}(R)$ is generated by elements of the type $$1 + aE_1E_i, \hskip 2mm 1+ aE_iE_1, \hskip 2mm 1+ aF_1F_i,
\hskip 2mm 1 + aF_iF_1, \hskip 2mm i\neq 1.$$

\vskip 2mm
\begin{defn}
An elementary matrix has $1$'s on the diagonal and at most one other entry is nonzero. Let $E_{ij}(x)$ denote the elementary matrix
with $x$ in the $(i,j)$ position. The group generated by $n \times n$ elementary matrices is denoted by $E_n(R)$.
\end{defn}
\vskip 5mm

\begin{thm}
 $Epin_6(R) = E_4(R)$.
 \end{thm}

\vskip 2mm
\begin{proof}
 Let $v=(a_1,a_2,a_3)$ and $w =(b_1,b_2,b_3)$. Then 
 $$ S(v,w) =  \begin{pmatrix}
a_1 & 0 & a_2 & a_3\\
0   & a_1 & -b_3 & b_2  \\
-b_2 & a_3 & b_1 & 0 \\
-b_3 & -a_2 & 0  & b_1 \end{pmatrix} .
$$

The group $EG_2(R)$ contains following matrices and their transposes :
\[E_{13}(x) = 1 + xE_1E_2,\]
 \[E_{14}(x) = 1+ xE_1E_3, \]
\[E _{24}(x) = 1 + xE_1F_2, \]
\[E_{23}(x) = 1 - xE_1F_3.\]
Observe that the generators of $EG_2(R)$
$$1 + xE_1E_i, \hskip 2mm 1+ xE_iE_1, \hskip 2mm 1+ xF_1F_i,
\hskip 2mm 1 + xF_iF_1, \hskip 2mm i\neq 1$$ are all elementary matrices. Therefore $EG_2(R) \subseteq E_4(R)$.
\vskip 2mm

 We also have $$E_1^T =E_1, \hskip 2mm E_i^T = -F_i.$$ Thus if $g \in EG_2(R)$, then $g^T \in EG_2(R)$.
Therefore to conclude that $E_4(R) \subseteq EG_2(R)$ it suffices to check that every elementary matrix with a nonzero entry above the diagonal falls into $EG_2(R)$. 
\vskip 2mm
It remains to be checked that $E_{12}(x)$
and $ E_{34}(x)$ also fall in $EG_2(R)$. Indeed we have
\[ E_{12}(x) = [E_{13}(x), E_{32}(1)]\]
and 
\[ E_{34}(x) = [E_{31}(x), E_{14}(1)]. \qedhere\]
\end{proof}
\marginnote[-3 in]{
The generators of $EG_{n-1}(R)$ mentioned above, i.e. $$1 + xE_1E_i, \hskip 2mm 1+ xE_iE_1, \hskip 2mm 1+ xF_1F_i,
\hskip 2mm 1 + xF_iF_1, \hskip 2mm i\neq 1$$ are the same as the ``top and bottom'' matrices
$E(e_i)(x)^{tb}$, $E(e_i^*)(x)^{tb}$ (defined in \cite{JR}). 
These ``top and bottom'' matrices generate of the group $EUm_{n-1}(R)$, the group 
generated by the Suslin matrices $S(e_1\epsilon, e_1{\epsilon^T}^{-1})$ 
where $\epsilon \in E_n(R)$ (see \cite{JR3}, Proposition 2.6).
Hence for odd $n$, we have $$EUm_{n-1}(R) = EG_{n-1}(R) \cong Epin_{2n}(R).$$
\vskip 2mm
Moreover since $Epin_{2n}(R)$ maps onto $EO_{2n}(R)$ it follows that $EUm_{n-1}(R)$ maps onto $EO_{2n}(R)$, the key result in developing the Quillen-Suslin theory
for the pair ($SUm_{n-1}(R),EUm_{n-1}(R)$) where $SUm_{n-1}(R)$ is the group generated by Suslin matrices of unit length.}\label{rmk00}

\end{NoHyper}

\end{document}